\documentclass[a4paper,12pt]{article}

\usepackage{amssymb}
\usepackage{epsfig}
\usepackage{amsfonts}
\usepackage{amsmath}
\usepackage{euscript}
\usepackage{amscd}
\usepackage{amsthm}
\DeclareMathAlphabet{\mathpzc}{OT1}{pzc}{m}{it}

\newtheorem{thm}{Theorem}[section]
\newtheorem{lem}[thm]{Lemma}
\newtheorem{prop}[thm]{Proposition} 
\newtheorem{cor}[thm]{Corollary}
\newtheorem{rem}[thm]{Remark}
\newtheorem{ex}[thm]{Example}

\newcommand{\bR}{\mathbb R}
\newcommand{\bC}{\mathbb C}
\newcommand{\bN}{\mathbb N}
\newcommand{\A}{\mathbb A}

\newcommand{\td}{\operatorname{tr.deg}}

 \title{An algebraic characterization of the affine three space}
 \author{Nikhilesh Dasgupta and Neena Gupta\\
 {\small{\it Statistics and Mathematics  Unit, Indian Statistical Institute,}}\\
 {\small{\it 203 B.T. Road, Kolkata 700 108, India}}\\
 {\small{\it e-mail:  its.nikhilesh@gmail.com}}\\
 {\small{\it e-mail : neenag@isical.ac.in, rnanina@gmail.com}}}

\begin{document}
\date{}
\maketitle

\abstract{In this paper we give algebraic characterizations of the affine $2$-space and the affine $3$-space over an 
algebraically closed field of characteristic zero, using a variant of the Makar-Limanov invariant.

\smallskip

 \noindent
 {\small {{\bf Keywords}. Polynomial Rings, Locally Nilpotent Derivation, Makar-Limanov Invariant.}
 
 \noindent
 {\small {{\bf 2010 MSC}. Primary: 14R10; Secondary: 13F20, 13A50, 13N15, 14R20.}}
 }

\section{Introduction}

Let $R$ be a ring and $n(\geqslant 1)$ be an integer. For an $R$-algebra $A$, we use the notation $A=R^{[n]}$ to 
denote that $A$ is isomorphic to a polynomial ring in $n$ variables over $R$. 

One of the important problems in affine algebraic geometry is to find a useful characterization of an affine $n$-space. 
This ``Characterization Problem'' is closely related to other challenging problems on the affine space like the 
``Cancellation Problem''. For instance, if $k$ is an algebraically closed field
of characteristic zero, $k^{[1]}$ is the only one-dimensional UFD with trivial units. This is an algebraic 
characterization. It immediately solves the Cancellation Problem: $A^{[1]} =k^{[2]} \implies A= k^{[1]}$. 
If $k=\bC$,  then the affine line $\A^{1}_{\bC}$ is the only acyclic normal curve, a topological characterization.

In his attempt to solve the Cancellation Problem, C.P. Ramanujam obtained a remarkable topological
characterization of the affine plane $\bC^2$ (\cite{R}). 
Later M. Miyanishi (\cite{MN1}) obtained an algebraic characterization of the polynomial ring
$k^{[2]}$ over an algebraically closed field of characteristic zero, 
which was used by  T. Fujita, M. Miyanishi and T. Sugie (\cite{Fu}, \cite{MS})
to solve the Cancellation Problem for the affine plane. 
Since then, there have been several attempts to give a characterization of 
$k^{[3]}$. 
Remarkable results were obtained by Miyanishi \cite{MN2} and 
 Kaliman \cite{K}. These results involved some topological invariants.  
In this paper, we will use a variant of the Makar-Limanov invariant (defined below), to give  
new algebraic characterizations of $k^{[2]}$ and $k^{[3]}$.

 We first recall the Makar-Limanov invariant and its variant. 
Let $k$ be a field of characteristic zero and $B$ an affine $k$-domain.
 The set of locally nilpotent 
$k$-derivations of $B$ is denoted by $LND(B)$. 
We denote the kernel of a locally nilpotent derivation $D$ by $Ker$ $D$. 
The {\it Makar-Limanov invariant} of $B$, denoted 
by $ML(B)$, is defined to be 
$$
ML(B):=\bigcap_{D \in LND(B)} Ker~D.
$$ 
The Makar-Limanov invariant has been a powerful 
tool for solving some major problems in affine algebraic geometry
like the Linearization Problem \cite[pp. 195--204]{F}. 
When $k$ is an algebraically closed field of characteristic zero, the Makar-Limanov invariant also gives a characterization  of 
$k^{[1]}$, i.e., for an affine 
$k$-domain $B$ with $\td_{k}B=1$, $B=k^{[1]}$ if and only if $ML(B)=k$ (cf. Theorem \ref{cml2}). 
However, the triviality of the Makar-Limanov invariant alone does not characterize
the affine $2$-space (i.e., $\dim B=2$ and $ML(B)=k$ $\nRightarrow$ $B= k^{[2]}$; cf. Theorem \ref{dan}). 

In this paper, we show that under the additional condition that $B$ has a locally nilpotent derivation
$D$ ``with slice'' (i.e., $1 \in {\rm Im}(D)$), the condition $ML(B)=k$ does imply that $B= k^{[2]}$ (if $\dim B= 2$)
and $B= k^{[3]}$ (if $\dim B= 3$ and $B$ is a UFD). 
In fact, we show that Theorem \ref{cml2} can indeed be extended to dimensions $2$ and $3$, if we replace the condition
``$ML(B)=k$'' by ``$ML^*(B)=k$'', where $ML^*(B)$ is an invariant which we define below.


Consider the subset $LND^{*}(B) \subseteq LND(B)$ defined by 
$$
LND^{*}(B)=\{D \in LND(B)~|~Ds=1~ \text{for some} ~ s \in B\}.
$$ 
Then we define 
$$
ML^{*}(B):= \bigcap_{D \in LND^{*}(B)} Ker~D.
$$ 
This invariant is introduced by G. Freudenburg  in \cite[p. 237]{F}.
We shall call it the Makar-Limanov--Freudenburg invariant or ML-F invariant.
If $LND^{*}(B)=\emptyset$, we define 
$ML^{*}(B)$ to be $B$. 
Note that if $ML^{*}(B)= k$ then automatically $ML(B)=k$. 
Also note that $ML^{*}(k^{[n]})=ML(k^{[n]})=k$ for each $n \geq 1$. 

%
Our main results are: 
\bigskip

\noindent
{\it {\bf Theorem 3.8.}
Let $k$ be a field of characteristic zero and $B$ a two-dimensional affine $k$-domain. Then the following are equivalent:
\begin{enumerate}
 \item [\rm (I)]$B=k^{[2]}$.
 \item [\rm (II)]$ML^{*}(B)=k$.
 \item [\rm (III)]$ML(B)=k$ and $ML^{*}(B) \neq B$.
\end{enumerate}
}
\bigskip 

\noindent
{\it {\bf Theorem 4.6.}
Let $k$ be an algebraically closed field of characteristic zero and $B$ an affine $k$-domain such that $B$ is a UFD and 
dim $B=3$. Then the following are equivalent:
\begin{enumerate}
 \item [\rm (I)]$B=k^{[3]}$.
 \item [\rm (II)]$ML^{*}(B)=k$.
 \item [\rm (III)]$ML(B)=k$ and $ML^{*}(B) \neq B$.
\end{enumerate}
}

In Section $2$, we recall a few preliminary definitions and results; in Sections $3$ and $4$, we prove our main theorems on 
characterizations of $k^{[2]}$ and $k^{[3]}$ respectively. In section 4, we also give a classification of three-dimensional affine factorial domains
for which Makar-Limanov invariant and Makar-Limanov--Freudenburg invariant are same and in Section $5$ we present a few examples.


\section{Preliminaries}
{\bf Notation:}

By a ring, we will mean a commutative ring with unity. We denote the group of units of a ring $R$ by $R^{*}$. 
For a ring $A$ and a non-zerodivisor 
$f \in A$, we use the notation $A_f$ to denote the localisation of $A$ with respect to the multiplicatively closed 
set $\{1,f,f^2,\dots\}$. We denote the Krull dimension of a ring $B$ by dim $B$. For integral domains $A \subseteq B$, the 
transcendence degree of the field of fractions of $B$ over that of $A$ is denoted by $\td_{A}{B}$. 
Capital letters like $X,Y,Z,T,U,V$ will be used as indeterminates over respective 
ground rings; thus, $k[X,Y,Z]=k^{[3]}$, $R[U,V]=R^{[2]}$, etc. The notation $ML(B)$ and $ML^{*}(B)$ have been defined 
in Section 1.

\noindent
{\bf Definitions:}

A subring $A$ of an integral domain $B$ is defined to be {\it inert} in $B$ if, given non-zero 
$f,g \in B$, the condition $fg \in A$ implies $f \in A$ and $g \in A$.
One can see that an inert subring of a UFD is a UFD and intersection 
of inert subrings is again inert. 
If $A$ is an inert subring of $B$, then $A$ is algebraically closed in $B$; 
further if $S$ is a multiplicatively closed set in $A$ then $S^{-1}A$ is an inert subring of $S^{-1}B$.

A non-zero locally nilpotent derivation $D$ on $B$ is said to be {\it reducible} if there exists a non-unit $b \in B$ such that 
$DB \subseteq bB$; otherwise $D$ is said to be {\it irreducible}. 

We say two locally nilpotent derivations $D_1$ and $D_2$ 
$\in LND(B)$ are {\it distinct} if $Ker$ $D_1 \neq Ker$ $D_2$. 

An affine $k$-domain $B$ is defined to be 
{\it rigid} if it does not have any non-zero locally nilpotent derivation. Thus for a rigid ring $B$, $ML(B)=ML^{*}(B)=B$. 
$B$ is defined to be {\it semi-rigid} if there exists a non-zero locally nilpotent derivation 
$D$ on $B$ such that $LND(B)=\{ fD ~ | ~ f \in Ker ~D\}$. 
Thus for an affine $k$-domain $B$, with $LND(B)\neq \{0\}$, $B$ is semi-rigid
if and only if $ML(B)= Ker~D$ for all non-zero $D  \in  LND(B)$.

An element $s \in B$ is called a {\it slice} if $D(s)=1$, and a 
{\it local slice} if $D(s) \in A$ and $D(s) \neq 0$.

\medskip

We shall use the following necessary and sufficient criterion, due to Nagata, for an integral domain to be a UFD 
\cite[Theorem 20.2]{M}.
\begin{lem}\label{nag}
Let $R$ be a Noetherian domain. If there exists a prime element $x$ in $R$ such that $R_{x}$ is a UFD, then $R$ is a UFD.
\end{lem}

We now quote a well-known result (for a reference, see the proof of \cite[Lemma 2.8]{F}). 
\begin{lem}\label{pic}
Let $k$ be an algebraically closed field of characteristic zero and $C$ be an affine UFD over $k$ of dimension one. Then  
$C=k[t,\frac{1}{p(t)}]$, where $k[t]=k^{[1]}$ and $p(t) \in k[t] \setminus \{ 0\}$. As a consequence,
if $C^{*}=k^{*}$, then $C=k^{[1]}$.
\end{lem}

Let $k$ be a field of characteristic zero, $B$ a $k$-domain and $D$ an element of $LND(B)$ with a local slice $r \in B$. Let $t= D(r)$ . 
The Dixmier map induced by $r$ is defined to be  
the $k$-algebra homomorphism 
$\pi_r:B \rightarrow B_t$, given by 
\begin{equation*}
\pi_r(f)=\sum_{i\geqslant 0}\frac{(-1)^i}{i!}D^{i}f\frac{r^i}{t^i}.
\end{equation*}

The following result is known as the Slice Theorem \cite[Corollary 1.22]{F}.
\begin{thm}\label{st}
 Let $k$ be a field of characteristic zero and $B$ a $k$-domain with $LND^{*}(B) \neq \emptyset$. 
 Let $D \in LND(B)$ admit a slice $s\in B$ and let $A=Ker$ $D$. Then 
 \begin{enumerate}
  \item[\rm (a)]$B=A[s]$ and $D=\frac{\partial}{\partial s}$.
  \item[\rm (b)]$A=\pi_{s}(B)$ and $Ker$ $\pi_{s}=sB$.
  \item[\rm (c)]If $B$ is affine, then $A$ is affine.
 \end{enumerate}
\end{thm}

The following lemma states some basic properties of locally nilpotent derivations of an affine domain \cite[1.1]{DF}.
\begin{lem}\label{inert}
Let $k$ be a field of characteristic zero and $B$ be an affine $k$-domain. Let $D\in LND(B)$ and $A:=Ker$ $D$. Then 
the following hold:
\begin{enumerate}
 \item [\rm (i)]$A$ is an inert subring of $B$.
 \item [\rm (ii)]For any multiplicatively closed subset $S$ of $A \setminus \{0\}$, $D$ extends to a locally nilpotent 
 derivation of $S^{-1}B$ with kernel $S^{-1}A$ and $B \cap S^{-1}A=A$.
 \item [\rm (iii)]Moreover, if $D$ is non-zero, then $\td_{A}B=1$.
\end{enumerate}
As a consequence, $ML(B)$ and $ML^*(B)$ are inert subrings of $B$ and hence are algebraically closed in $B$.
\end{lem}

The following result ensures that $LND^{*}(B) \neq \emptyset$ whenever $B$ is a two-dimensional factorial affine domain 
over an algebraically closed field $k$ of characteristic zero with $LND(B) \neq \{0\}$ \cite[Lemma 2.9]{F}.
\begin{lem}\label{Fufd}
Let $k$ be an algebraically closed field of characteristic zero and $B$ 
an affine $k$-domain such that $B$ is a UFD and dim $B=2$. Then every non-zero irreducible element of $LND(B)$ has a 
slice. 
\end{lem}

We now state an important result for rigid domains by Crachiola and Makar-Limanov \cite[Theorem 3.1]{CML1}. 
\begin{thm}\label{cml1}
Let $k$ be a field of characteristic zero, $C$ an affine $k$-domain and $C[T]=C^{[1]}$. Then the following hold:
\begin{enumerate}
 \item [\rm (i)]$ML(C[T]) \subseteq ML(C)$.
 \item [\rm (ii)]$C$ is rigid if and only if $ML(C[T])=C$.
\end{enumerate}
\end{thm}

The following result gives a characterization of $k^{[1]}$ in terms of the Makar-Limanov invariant \cite[Lemma 2.3]{CML2}.
\begin{thm}\label{cml2}
Let $k$ be a field of characteristic zero and $A$ an affine $k$-domain with \\ $\td_{k}A=1$ such that $k$ is algebraically 
closed in $A$. Then $A=k^{[1]}$ if it has a non-zero locally nilpotent derivation. 
\end{thm}

We now recall a result on the Makar-Limanov invariant of Danielewski surfaces \cite[Theorem 9.1]{F}.
\begin{thm}\label{dan}
Let $k$ be an algebraically closed field of characteristic zero and $$B:=\frac{k[X,Y,Z]}{(X^{n}Z-p(Y))}$$ where 
 $n \in \bN$ and $p(Y) \in k[Y]$. Let $x$ be the image of $X$ in $B$. Then the following hold:
\begin{enumerate}
 \item [\rm (i)]If $n=1$ or if $\deg p(Y)=1$, then $ML(B)=k$.
 \item [\rm (ii)]If $n \geq 2$ and $\deg p(Y) \geq 2$, then $ML(B)=k[x]$. Moreover, $Ker$ $D=k[x]$ for every 
  non-zero $D \in LND(B)$.
\end{enumerate}
\end{thm}

\section{A characterization of $k^{[2]}$}
In this section we will describe an algebraic characterization of $k^{[2]}$ over a field $k$ of characteristic zero 
(Theorem \ref{a2}). We also investigate properties of a two-dimensional affine $k$-domain (say $B$) such that 
$ML(B)=ML^{*}(B)$.
We first begin with a few general lemmas. 


\begin{lem}\label{lem1}
Let $k$ be a field of characteristic zero, $C$ an affine $k$-domain and $C[T]=C^{[1]}$. 
Then the following hold:
\begin{enumerate}
 \item [\rm (i)]$ML^{*}(C[T]) \subseteq ML(C)$.
 \item [\rm (ii)]$C$ is rigid if and only if $ML^{*}(C[T])=C$.
\end{enumerate}
\end{lem}

\begin{proof}
(i) Given $D \in LND(C)$, we can extend $D$ to $\tilde{D} \in LND^{*}(C[T])$ by $\tilde{D}T=1$. Then: 
$$ML^{*}(C[T]) \subseteq Ker \text{~} \tilde{D} \cap Ker \text{~} \frac{\partial}{\partial T}=Ker \text{~} \tilde{D} \cap C 
=Ker \text{~} D.$$
Thus, for any $D\in LND(C)$, we have $ML^*(C[T]) \subseteq {\text{Ker~}}D$, and hence  $ML^*(C[T])  \subseteq ML(C)$.

(ii) Now suppose $ML^{*}(C[T])=C$. Then part (i) implies $C \subseteq ML(C)$, so $C$ is rigid. Conversely, if $C$ is 
rigid, then by Theorem \ref{cml1} and part (i) we have: 
$$C=ML(C[T]) \subseteq ML^{*}(C[T]) \subseteq ML(C)=C.$$ Hence $ML^{*}(C[T])=C$.
\end{proof}

Note that for an arbitrary affine $k$-domain $B$ of dimension one, 
$ML(B)=ML^{*}(B)$.  We have the following result on the equality 
of $ML(B)$ and $ML^{*}(B)$ for affine domains of dimension greater than one.
\begin{lem}\label{lem2}
Let $k$ be a field of characteristic zero and $B$ be an affine $k$-domain of dimension $n\ge 2$. If $\td_{k}ML^{*}(B)=n-1$, 
then $ML(B)=ML^{*}(B)$, $B=C^{[1]}$ for some rigid subring $C$ of $B$ and $B$ is a semi-rigid ring.
\end{lem}

\begin{proof}
 Since $ML^{*}(B) \neq B$, by Theorem \ref{st}, there exist an $(n-1)$-dimensional subring 
$C$ of $B$ such that $B= C^{[1]}$ and $ML^{*}(B) \subseteq C$. 
Since $\td_{k}ML^{*}(B)=n-1 =\td_{k}C$ and both $ML^*(B)$ and $C$ are algebraically closed in $B$, we have 
$ML^{*}(B) = C$. Hence, by Lemma \ref{lem1} (ii),  $C$ is  a rigid ring. 
Therefore, by Theorem \ref{cml1}, $ML(B)=C= ML^{*}(B)$. 

Let $D (\neq 0) \in LND(B)$ and $A= Ker~D$. Since $ML(B)\subseteq A$, $\td_{k}ML(B) = n-1= \td_k A$ and both 
$ML(B)$ and $A$ are algebraically closed in $B$, we have $A=ML(B)$, i.e., 
$Ker~D= ML(B)$ for all $D (\neq 0) \in LND(B)$. Thus  $B$ is semi-rigid.
\end{proof}

\begin{lem}\label{alrigid}
Let $k$ be a field of characteristic zero and $B$ an affine $k$-domain such that $B$ is a semi-rigid ring.  
Then the following are equivalent:
\begin{enumerate}
 \item [\rm (I)]$ML(B)=ML^{*}(B)$.
 \item [\rm (II)]There exists a $k$-subalgebra $C$ of $B$ such that $C$ is rigid and $B=C^{[1]}$.
 \item [\rm (III)]$ML(B)$ is rigid and $B={ML(B)}^{[1]}$.
\end{enumerate}
\end{lem}

\begin{proof}
${\rm (I)} \Rightarrow {\rm (II)}$
Let $D$ be a non-zero locally nilpotent derivation such that 
 $LND(B)=\{ fD ~ | ~ f \in Ker$ $D\}$ and $C:=Ker$ $D$. Since $ML^{*}(B)=ML(B)$, $D$ has a slice, say $s$. 
 Thus by Theorem \ref{st}, $B=C[s]=C^{[1]}$. 
 If $LND(C) \neq \{0\}$ and $d (\neq 0) \in LND(C)$, then $d$ extends to 
 $\tilde{d} \in LND(B)$ with $\tilde{d}(s)=0$. But then $Ker~\tilde{d}\neq C$, contradicting that $B$ is 
semi-rigid. Thus $LND(C) = \{0\}$, i.e., $C$ is rigid. 

${\rm (II)} \Rightarrow {\rm (III)}$ Trivial.

${\rm (III)} \Rightarrow {\rm (I)}$ Follows from Lemma \ref{lem1} (ii).
\end{proof}

As a consequence we have the following sufficient condition for equality of the two invariants $ML(B)$ and $ML^{*}(B)$ 
for a two-dimensional affine domain $B$. 
\begin{lem}\label{2dim}
Let $k$ be a field of characteristic zero and $B$ a two-dimensional affine $k$-domain. Suppose that $ML^{*}(B) \neq B$. 
Then $ML^{*}(B)=ML(B)$.
\end{lem}

\begin{proof}
Clearly $\td_{k}ML^{*}(B) \le 1$. Since $ML(B)$ and $ML^{*}(B)$ are algebraically closed subrings of $B$ and $ML(B) 
\subseteq ML^{*}(B)$, it is enough to  consider the case $\td_{k}ML^{*}(B) = 1$.  The result now follows from Lemma \ref{lem2}.
\end{proof}

Example \ref{danex} presents a two-dimensional affine domain $B$ for which $ML(B) \subsetneqq ML^{*}(B)=B$. However the 
following consequence of Lemma \ref{Fufd} shows that such an example is not possible when $B$ is a UFD. 

\begin{cor}\label{ufddim2}
Let $k$ be an algebraically closed field of characteristic zero and $B$ 
a two-dimensional affine $k$-domain such that $B$ is a UFD.
Then $ML(B)=ML^{*}(B)$.
\end{cor}

\begin{proof}
If $ML(B) = B$, then by definition, we have $ML^*(B) = B$. 
Now if $ML(B) \neq B$, then by Lemma \ref{Fufd}, $LND^{*}(B) \neq \emptyset$, i.e., $ML^{*}(B) \neq B$ and 
hence $ML(B)=ML^{*}(B)$ by Lemma \ref{2dim}.
\end{proof}

%
%
We have the following properties of a two-dimensional affine domain whenever the two invariants 
$ML(B)$ and $ML^*(B)$ are same. 
\begin{prop}\label{dim2}
Let $k$ be a field of characteristic zero and $B$ a two-dimensional affine $k$-domain. If $ML(B)=ML^{*}(B)$, then $ML(B)$ is 
rigid and $B$ is a polynomial ring over $ML(B)$.
\end{prop}

\begin{proof}
 Suppose $\td_{k}ML(B)=2$. Then $B$ is rigid and $ML(B)=B$. 
 
 Now suppose $\td_{k}ML(B)=1$. Then by Lemma \ref{lem2}, $B$ is semi-rigid and hence by Lemma \ref{alrigid} (III),  
 $ML(B)$ is rigid and $B={ML(B)}^{[1]}$.
 
 Finally, suppose $\td_{k}ML(B)=0$. Then we have $ML(B)=ML^{*}(B)=L$, where $L$ is the algebraic closure of $k$ in $B$. 
 Since $ML^{*}(B) \neq B$, by Theorem \ref{st}, 
 there exists a one-dimensional subring $C$ of $B$ such that $B=C^{[1]}$. 
 Now $C$ is not rigid, otherwise by Theorem \ref{cml1}, $ML(B)=ML({C}^{[1]})=C$ contradicting that $ML(B)=L$. Hence, 
 by Theorem \ref{cml2}, $C=L^{[1]}$. Thus $B=L^{[2]}$.
\end{proof}

\begin{rem}
{\em The proof of Proposition \ref{dim2} shows that for a two-dimensional affine domain $B$ over a field $k$ of 
characteristic zero satisfying $ML(B)=ML^{*}(B)$ we have the following three cases: \\
(i) If $\td_{k}ML(B)=2$, then $B$ is rigid and $ML(B)=B$. \\
(ii) If $\td_{k}ML(B)=1$, then $B$ is semi-rigid and $B=C^{[1]}$ where $C=ML(B)$ is rigid. \\
(iii) If $\td_{k}ML(B)=0$, then $B=L^{[2]}$, where $L=ML(B)$ is the algebraic closure of $k$ in $B$.
}
\end{rem}

As a consequence of Lemma  \ref{2dim} and Proposition \ref{dim2}, we have the following characterization of the affine 
$2$-space.
\begin{thm}\label{a2}
Let $k$ be a field of characteristic zero and $B$ a two-dimensional affine $k$-domain. 
Then the following are equivalent:
\begin{enumerate}
 \item [\rm (I)]$B=k^{[2]}$.
 \item [\rm (II)]$ML^{*}(B)=k$.
 \item [\rm (III)]$ML(B)=k$ and $ML^{*}(B) \neq B$.
\end{enumerate}
\end{thm}

\begin{proof}
Clearly ${\rm (I)} \Rightarrow {\rm (II)} \Rightarrow {\rm (III)}$. We now show that 
${\rm (III)} \Rightarrow {\rm (I)}$. Since $ML^{*}(B) \neq B$, by Lemma \ref{2dim} we have $ML^{*}(B)=ML(B)$.
Therefore, as $ML(B)=k$, we have, by Proposition \ref{dim2}, $B=k^{[2]}$.  
\end{proof}

\section{A characterization of $k^{[3]}$}
In this section we will describe an algebraic characterization of $k^{[3]}$ over an algebraically closed field $k$ of characteristic 
zero (Theorem \ref{a3}).
We also investigate properties of a three-dimensional affine $k$-domain (say $B$) over an algebraically closed
field of characteristic zero for which $ML(B)=ML^{*}(B)$. 

We first state a result for a polynomial ring in two variables over a one-dimensional affine UFD.
\begin{lem}\label{ufd}
Let $k$ be an algebraically closed field of characteristic zero, $R$ a one-dimensional affine UFD and $B:=R^{[2]}$. 
Then either  $B=k^{[3]}$ or $ML(B)=R$.
\end{lem}

\begin{proof}
By Lemma \ref{pic}, $R=k[t, \frac{1}{p(t)}]$ where $k[t]=k^{[1]}$ and $p(t) \in k[t] \setminus \{0\}$. Now, either 
$p(t) \in k$ or $p(t) \notin k$. If $p(t) \in k$, then $R=k^{[1]}$ and $B=k^{[3]}$. If $p(t) \notin k$, then 
$ML(B)=R$ since $p(t) \in ML(B)$ and $ML(B)$ is an inert subring of $B$.
\end{proof}

The next result shows that if a three-dimensional affine UFD $B$ admits two non-zero distinct locally nilpotent derivations 
with slices, then there exists a $k$-subalgebra $R$ of $B$, such that $B=R^{[2]}$. Example \ref{dim4ex} shows that such a 
result does not extend to a four-dimensional affine UFD.
\begin{lem}\label{dim3}
Let $k$ be an algebraically closed field of characteristic zero and $B$ an affine $k$-domain such that $B$ is a UFD and 
dim $B=3$. If $B$ admits two non-zero distinct locally nilpotent derivations 
with slices, then there exists a $k$-subalgebra $R$ of $B$, such that $R$ is a UFD and $B=R^{[2]}$. 
\end{lem}

\begin{proof}
Let $D_1$ and $D_2$ be two non-zero distinct locally nilpotent derivations  of $B$ with slices $s_1, s_2$. 
Let $Ker$ $D_i=C_i$ for $i=1,2$. Then, by Theorem \ref{st}, $B=C_i[s_i]={C_i}^{[1]}$ for each $i$. Now 
$ML(B) \subseteq ML^{*}(B) \subseteq C_{1} \cap C_{2} \subsetneqq C_{i}$. It follows that $C_i$ is not rigid, otherwise by Theorem \ref{cml1}, 
$ML(B)=ML({C_{i}}^{[1]})=ML(C_{i})=C_{i}$. Since $C_{1}$ is an inert subring of the UFD $B$, $C_{1}$ is a UFD. 
As $C_1$ is not rigid, by Lemma \ref{Fufd}, $C_{1}$ has a locally nilpotent derivation with a slice and therefore by Theorem \ref{st}, 
$C_{1}=R^{[1]}$ for some $k$-subalgebra $R$ of $C_{1}$. Hence $B=R^{[2]}$.
As $R$ is an inert subring of the UFD $C_1$,  $R$ is a UFD. 
\end{proof}

The following result describes a classification of three-dimensional factorial affine domains $B$ for which 
$ML(B)=ML^{*}(B)$.
\begin{prop}\label{dim3new}
Let $k$ be an algebraically closed field of characteristic zero and $B$ a three-dimensional affine UFD over $k$. If 
$ML(B)=ML^{*}(B)$, then $ML(B)$ is a rigid UFD and $B$ is a polynomial ring over $ML(B)$.
\end{prop}

\begin{proof}
 Suppose $\td_{k}ML(B)=3$, then $B$ is rigid and $ML(B)=B$. Now suppose $\td_{k}ML(B)=2$. Then by Lemmas 
 \ref{lem2} and \ref{alrigid}, $B={ML(B)}^{[1]}$ and $ML(B)$ is rigid. Since $B$ is a UFD, $ML(B)$ is also a UFD. 
 
 Now suppose $\td_{k}ML(B) \leq 1$. Then $B$ admits two non-zero distinct 
locally nilpotent derivations of $B$ with slices.  Hence, by Lemma \ref{dim3}, there exists a one-dimensional 
$k$-subalgebra $S$ of $B$ such that $B=S^{[2]}$. Since $B$ is a UFD, $S$ is a UFD.
If $\td_{k}ML(B)=1$, then $B\neq k^{[3]}$ and hence $S\neq k^{[1]}$. Hence, 
by Theorem \ref{cml2}, $S$ is rigid, and by Lemma \ref{pic}, 
there exists $t \in B$ such that $S=k[t,\frac{1}{p(t)}]$, where $k[t]=k^{[1]}$ and $p(t) \in k[t] \setminus k$. 
In particular, $k^{*} \subsetneqq B^{*}$. Thus $ML(B)=S$ and $B=S^{[2]}$. Again if 
$\td_{k}ML(B)=0$, then $ML(B)=k \neq S$. Hence by Lemma \ref{ufd}, $B=k^{[3]}$.
\end{proof} 

\begin{rem}
{\em The proof of Proposition \ref{dim3new} shows that for a three-dimensional affine UFD $B$ over an algebraically 
closed field $k$ of characteristic zero satisfying $ML(B)=ML^{*}(B)$ we have the following four cases: \\
(i) If $\td_{k}ML(B)=3$, then $B$ is rigid. \\
(ii) If $\td_{k}ML(B)=2$, then $B=C^{[1]}$, where $C$ is rigid. \\
(iii) If $\td_{k}ML(B)=1$, then $B=S^{[2]}$, where $S=k[t,\frac{1}{p(t)}]$ for some $p(t) \in k[t] \setminus k$. \\
(iv) If $\td_{k}ML(B)=0$, then $B=k^{[3]}$.
}
\end{rem}

The following result shows that for a three-dimensional factorial affine domain over an algebraically closed field, 
the equality of $ML(B)$ and $ML^{*}(B)$ holds whenever $ML^{*}(B) \neq B$.
\begin{lem}\label{ufddim3}
Let $k$ be an algebraically closed field of characteristic zero and $B$ an affine $k$-domain such that $B$ is a UFD and 
dim $B=3$. If $ML^{*}(B) \neq B$, then $ML(B)=ML^{*}(B)$. 
\end{lem}

\begin{proof}
Since $ML^{*}(B) \neq B$, we have $\td_{k}ML^{*}(B)\le 2$.
If $\td_{k}ML^{*}(B)=2$, then the result follows from Lemma \ref{lem2}.

Now suppose $\td_{k}ML^{*}(B)=1$. Then, by Lemma \ref{dim3},  
there exists a $k$-subalgebra $R$ of $B$ such that $R$ is a one-dimensional UFD and $B=R^{[2]}$. 
Thus  $ML^{*}(B) \subseteq R$. As both $ML^{*}(B)$ and $R$ are algebraically closed in $B$ 
and have the same transcendence degree over $k$, 
we have $ML^{*}(B)=R$.
As $ML^{*}(B)\neq k$, $B \neq k^{[3]}$ and hence 
 $ML(B)=R$ by Lemma \ref{ufd}.   Thus $ML(B)=ML^{*}(B)$.
 
If $\td_{k}ML^{*}(B)=0$, then $ML^{*}(B)=k$ and hence $ML(B)=k= ML^*(B)$.
\end{proof}

We now state our main result.
\begin{thm}\label{a3}
Let $k$ be an algebraically closed field of characteristic zero and $B$ an affine $k$-domain such that $B$ is a UFD and 
dim $B=3$. Then the following are equivalent:
\begin{enumerate}
 \item [\rm (I)]$B=k^{[3]}$.
 \item [\rm (II)]$ML^{*}(B)=k$.
 \item [\rm (III)]$ML(B)=k$ and $ML^{*}(B) \neq B$.
\end{enumerate}
\end{thm}

\begin{proof}
Clearly ${\rm (I)} \Rightarrow {\rm (II)} \Rightarrow {\rm (III)}$. We now show that 
${\rm (III)} \Rightarrow {\rm (I)}$. Since $ML^{*}(B) \neq B$, by Lemma \ref{ufddim3}, $ML(B)=ML^{*}(B)$. Now by 
Proposition \ref{dim3new}, $B=k^{[3]}$. 
\end{proof}

\begin{rem}
{\em (i) The hypothesis that $ML^{*}(B) \neq B$ is necessary in Lemma \ref{ufddim3}. We will show that for 
a three-dimensional affine UFD $B$, containing an algebraically closed field of characteristic zero, it may happen that 
$ML^{*}(B)=B$ but $\td_{k}ML(B)$ is zero (Example \ref{threefold}), one (Example \ref{rkt}), or two (Example \ref{237}) 
i.e. $ML^{*}(B) \neq ML(B)$. 

(ii) Example \ref{danex3} will show that both the hypotheses ``$k$ is an algebraically closed field'' and ``$B$ is a UFD'' are 
needed for the implication (III) $\implies$ (I) in Theorem \ref{a3}. 
}
\end{rem}

\medskip

Lemma \ref{lem2} shows that there does not exist any three-dimensional affine $k$-domain $B$ such that
$ML(B)\subsetneqq ML^*(B)$ but $\td_{k}ML^{*}(B)=2$.
 However we pose the following question.
 
\medskip

\noindent
{\bf Question 4.8.} {\it Does there exist a three-dimensional affine $k$-domain $B$ over a field $k$ of characteristic zero
such that $ML(B)=k$ but $\td_{k}ML^{*}(B)=1$?}

\medskip

Note that Theorem \ref{a3} shows that Question 4.8 has negative answer when $k$ is an algebraically closed field and
 $B$ is a UFD. If the answer to Question 4.8 is negative in general then the implication (III) $\implies$ (II) will hold in 
 Theorem \ref{a3} even without the additional hypotheses that ``$k$ is an algebraically closed field'' and  ``$B$ is a UFD''. 

\section{Some examples}
In this section we shall present some examples to illustrate the hypotheses of the results stated earlier.
The following example shows that both the hypotheses  ``$k$ is algebraically closed'' and ``$B$ is a UFD'' are needed in Theorem 
\ref{a3}. 

\begin{ex}\label{danex3}
{\em
Let $k$ be a field of characteristic zero, $$R:=\frac{k[X,Y,Z]}{(XY-Z^{2}-1)} \text{~~and~~} B:=R[T].$$ 
Then the following hold:
\begin{enumerate}
 \item [\rm (i)] If $k$ is an algebraically closed field, then $B$ is not a UFD.
 \item [\rm (ii)] If $k=\bR$, then $B$ is a UFD.
  \item [\rm (iii)] $ML^{*}(B)=k$.
 \item [\rm (iv)] $B \neq k^{[3]}$.
\end{enumerate}
Thus the conditions (II) and (III) of Theorem \ref{a3} hold but not (I).

\begin{proof}
 Let $x$, $y$ and $z$ denote the images in $B$ of $X$, $Y$ and $Z$ respectively. 

(i)  One can see that $x$ is an irreducible element of $B$. Now if $k$ is an algebraically closed field, then clearly $x$ is 
not a prime element in $B$. 

(ii) Suppose $k=\bR$. Then $x$ is a prime element in $B$ and since $B[1/x](= k[x,1/x]^{[2]})$ is a UFD, we have $B$ is a UFD 
by Lemma \ref{nag}. 

(iii) Consider the two locally nilpotent $k$-derivations on $B$, say $D_1$ and $D_2$ given by 
$$D_1(x)=0, ~~D_1(y)=2z, ~~D_1(z)=x, ~~D_1(T)=1 ~~\text{and}$$ $$D_2(x)=2z, ~~D_2(y)=0, ~~D_2(z)=y, ~~D_2(T)=1.$$ Let 
$A_{i}=Ker$ $D_{i}$ for $i=1,2$. Then by Theorem \ref{st}, $$A_1=k[x,y-2zT+xT^{2}, z-xT] ~~\text{and}$$ 
$$A_2=k[y,x-2zT+yT^{2}, z-yT].$$ We now show that $A_1 \cap A_2=k$. Consider $A_1$ as a subring of 
$A_1[\frac{1}{x}]=k[x,\frac{1}{x},\frac{z}{x}-T]$ and $A_2$ as a subring of $A_2[\frac{1}{y}]=k[y,\frac{1}{y},\frac{z}{y}-T]$. 
Let $\alpha \in A_1 \cap A_2$ and $n:=T$-degree of $\alpha$. Then there exist elements $a_{i}(x) \in k[x,\frac{1}{x}]$ 
and $b_{j}(y) \in k[y,\frac{1}{y}]$ for $i,j \in \{0,1, \dots,n\}$ such that 
$$\alpha= \sum_{i=0}^{n}a_{i}(x){(\frac{z}{x}-T)}^{i}= \sum_{j=0}^{n}b_{j}(x){(\frac{z}{y}-T)}^{j}.$$ 
Comparing the coefficients of $T^{n}$ from the two expressions, we have ${(-1)}^{n}a_{n}(x)={(-1)}^{n}b_{n}(y) \in 
k[x,\frac{1}{x}] \cap k[y, \frac{1}{y}]=k$ (since $x$ and $y$ are algebraically independent over $k$). Again, comparing the 
coefficients of $z^{n}$ from the two expressions, we have $\frac{a_{n}(x)}{x^n}=\frac{b_{n}(y)}{y^n}$. Hence $n=0$ and 
consequently $\alpha \in k$. Thus $ML^{*}(B)=k$.

(iv) Let $\bar{k}$ denote the algebraic closure of $k$. Then $B {\bigotimes}_{k} \bar{k}$ is not a UFD by (i). Hence $B \neq k^{[3]}$.
\end{proof}
}
\end{ex}

We now present examples of  affine domains $B$ for which $ML(B) \subsetneqq ML^{*}(B)=B$.
 We first present an example for $\dim B=2$.  
By Corollary \ref{ufddim2}, such an example 
is not possible for two-dimensional factorial affine domains.

\begin{ex}\label{danex}
{\em Let $k$ be an algebraically closed field of characteristic zero, $n \geq 1$ be an integer and 
$p(Y) \in k[Y]$ be such that $\deg p(Y) \geq 2$. Let 
$$B:=\frac{k[X,Y,Z]}{(X^{n}Z-p(Y))}.$$ 
Let $x$ denote the image of $X$ in $B$. $B$ is not a UFD (since $x$ is irreducible but not a prime in $B$). We have
\begin{enumerate}
 \item [\rm (i)] If $n=1$, then $ML(B)=k$ by Theorem \ref{dan}(i) but $ML^{*}(B)=B$ by Theorem \ref{a2} (since $B \neq k^{[2]}$). 
 \item [\rm (ii)] If $n \geq 2$, then $ML(B)=k[x]$ and $B$ is a semi-rigid ring by Theorem \ref{dan}(ii) but $ML^{*}(B)=B$ 
 by Theorem \ref{st} (since $B \neq k^{[2]}$). 
\end{enumerate}
}
\end{ex}

\medskip
We now present examples of three-dimensional affine UFD $B$ for which $ML^{*}(B)=B$ but $ML(B) \subsetneqq ML^{*}(B)$. In the 
three examples $\td_{k}ML(B)$ is two, one and zero respectively.

\begin{ex}\label{237}
{\em Let $$R:=\frac{\bC[X,Y,Z]}{(X^{2}+Y^{3}+Z^{7})} \text{~~and~~} B:= \frac{R[U,V]}{(X^{2}U-Y^{3}V-1)}.$$ 
It has been proved by D.R. Finston and S. Maubach that $B$ is a semi-rigid 
UFD of dimension $3$ and $ML(B)=R$ \cite[Theorem 2]{FM}; in particular, $\td_{k}ML(B)=2$. But $ML^{*}(B)=B$ by 
Theorem \ref{st} (since $B \neq R^{[1]}$).
}
\end{ex}

\begin{ex}\label{rkt}
{\em Let $$B:=\frac{\bC[X,Y,Z,T]}{(X+X^{2}Y+Z^{2}+T^{3})},$$ be the well-known Russell-Koras threefold. 
Let $x$ denote the image of $X$ in $B$. Since $B_x$ is a UFD, by Lemma \ref{nag}, 
$B$ is a UFD. It has been proved by L.G. Makar-Limanov that $ML(B)=\bC[x]={\bC}^{[1]}$ \cite[Lemma 8]{ML}; in particular 
$\td_{k}ML(B)=1$. Since $B^{*}={\bC}^{*}$, we have $ML(B) \neq ML^{*}(B)$ by Proposition \ref{dim3new} and hence 
$ML^{*}(B)=B$ by Lemma \ref{ufddim3}. Note that the results in \cite{DZP} also imply that $ML^{*}(B)=B$.
} 
\end{ex}

\begin{ex}\label{threefold}
{\em Let $k$ be an algebraically closed field of characteristic zero and $$B:=\frac{k[X,Y,Z,T]}{(XY-ZT-1)}.$$ 
Let the images of $X$, $Y$, $Z$ and $T$ 
in $B$ be denoted by $x$, $y$, $z$ and $t$ respectively. Since $B_x$ is a UFD, by Lemma \ref{nag}, 
$B$ is a UFD. Moreover $B$ is regular. Consider four non-zero locally nilpotent derivations 
$D_{1}, ~ D_{2}, ~ D_{3} ~ \text{and} ~ D_{4}$ on $B$ given by
\begin{enumerate}
 \item [\rm (i)]$D_{1}x=0, \quad D_{1}y=z, \quad D_{1}z=0, \quad D_{1}t=x$. 
 \item [\rm (ii)]$D_{2}x=0, \quad D_{2}y=t, \quad D_{2}z=x,\quad D_{2}t=0$.
 \item [\rm (iii)]$D_{3}x=z, \quad D_{3}y=0, \quad D_{3}z=0, \quad D_{3}t=y$.
 \item [\rm (iv)]$D_{4}x=t, \quad D_{4}y=0, \quad D_{4}z=y, \quad D_{4}t=0$. 
\end{enumerate}
Let $Ker~D_{i}=A_{i}$ for each $i=1,2,3,4$. Now $k[x,z] \subseteq A_{1} \subseteq B$.
Since both $k[x,z]$ and $A_{1}$ are algebraically closed in $B$ and have the same transcendence degree over $k$, we have $A_{1}=k[x,z]$. 
Similarly $A_{2}=k[x,t]$, $A_{3}=k[y,z]$, $A_{4}=k[y,t]$ and $\bigcap_{i}A_{i}=k$. Thus $ML(B)=k$, i.e. $\td_{k}ML(B)=0$. 
But $B \neq k^{[3]}$ (since the Whitehead group $K_{1}(B) \neq k^{*}$) and it follows from Theorem \ref{a3} 
that $ML^{*}(B)=B$.
}
\end{ex}

We now present an example which shows that Theorem \ref{a3} does not extend to a four-dimensional affine regular UFD, i.e., a 
four-dimensional affine UFD $\widetilde{B}$ need not be $k^{[4]}$, even when $ML(\widetilde{B})=ML^{*}(\widetilde{B})=k$.  
We will follow the notation of Example \ref{threefold}. 
\begin{ex}\label{dim4ex1}
{\em Let $B$ be as in Example \ref{threefold} and $\widetilde{B}:=B[u]=B^{[1]}$. 
$\widetilde{B}$ is a regular UFD of dimension four. For each $i=1,2,3,4$, we extend the 
locally nilpotent derivation $D_{i}$ of $B$ to a locally nilpotent derivation $\widetilde{D_{i}}$ of $\widetilde{B}$, 
by defining $\widetilde{D_{i}}u=1$. Let 
$$\widetilde{D_{5}}=\frac{\partial}{\partial u}~\text{ and }Ker~\widetilde{D_{i}}=\widetilde{A_{i}}.$$ 
By Theorem \ref{st}, we have  
$$\widetilde{A_{1}}=k[x,z,y-zu,t-xu],$$ 
$$\widetilde{A_{2}}=k[x,t,z-xu,y-tu],$$ 
$$\widetilde{A_{3}}=k[y,z,x-zu,t-yu],$$ 
$$\widetilde{A_{4}}=k[y,t,x-tu,z-yu]~\text{ and }$$
$$\widetilde{A_{5}}=k[x,y,z,t].$$
Clearly $k[x,z+t-xu]\subseteq \widetilde{A_{1}} \cap \widetilde{A_{2}}$.
Since $k[x,z+t-xu]$ and $\widetilde{A_{1}} \cap \widetilde{A_{2}}$ are algebraically closed in $B[u]$ and they have the 
same transcendence degree over $k$, we have  
$$
\widetilde{A_{1}} \cap \widetilde{A_{2}}=k[x,z+t-xu].
$$
Similarly, 
$$
 \widetilde{A_{3}} \cap \widetilde{A_{4}}=k[y,z+t-yu].
$$
Again, 
$$
\widetilde{A_{5}} \cap k[x,z+t-xu]=k[x] \text{~and~}  \widetilde{A_{5}} \cap k[y,z+t-yu]=k[y].
$$
Hence $\bigcap_{i}\widetilde{A_{i}}=k$. Thus 
$ML^{*}(\widetilde{B})=ML(\widetilde{B})=k$. 
But $\widetilde{B}\neq k^{[4]}$ (for instance, $K_{1}(\widetilde{B})=K_{1}(B) \neq k^{*}$).
}
\end{ex}

The following example shows that Lemma \ref{dim3} need not be true for a four-dimensional affine UFD $B$, even if 
$ML(B)=ML^{*}(B)$.
\begin{ex}\label{dim4ex}
{\em Let $k$ be an algebraically closed field of characteristic zero and \\ 
$R= k[X,Y,Z]/(X^2+Y^3+Z^7)= k[x,y,z]$, where 
$x$, $y$ and $z$ denote the images of $X$, $Y$ and $Z$ in $R$. Let $C=R[U,V]/(xU-yV-1)=R[u,v]$, where $u$ and $v$ denote the 
images of $U$ and $V$ in $C$ and $B=C[T]=C^{[1]}$. Then the following hold.
\begin{enumerate}
 \item [\rm (i)]$B$ is a UFD of dimension $4$.
 \item [\rm (ii)]$ML(B)= ML^{*}(B)=R$.
 \item [\rm (iii)]$B \neq  R^{[2]}$.
 \item [\rm (iv)]$B \neq S^{[2]}$ for any $k$-subalgebra $S$ of $B$.
\end{enumerate}

\begin{proof}
(i) By Lemma \ref{nag}, $R$ and $C$ are UFDs. Hence $B$ is a UFD. Clearly dim $B=4$.

(ii) By \cite[Lemma 2]{FM}, $R\subseteq ML(B)\subseteq ML^{*}(B)$. 
Consider the $R$-linear derivations $\delta_1$ and $\delta_2$ on $B$ as follows:
\[
\delta_1(u)= y,~~ \delta_1(v)=x~~\text{ and }~~\delta_1(T)=1
\]
and  
\[
\delta_2(u)= yT,~~ \delta_2(v)=xT~~\text{ and }~~\delta_1(T)=1.
\]
Clearly they are locally nilpotent derivations with slices $T$.
By Theorem \ref{st}, $A_1:= Ker$ $\delta_1= R[u-yT, v-xT]$
and $A_2:= Ker$ $\delta_2= R[2u-yT^2, 2v-xT^2]$.
Then ${A_1}_x= R_x[v-xT]$ and ${A_2}_x= R_x[2v-xT^2]$
and the two rings $A_1$ and $A_2$ are clearly different. Therefore $A_1 \cap A_2 \subsetneqq A_2$. 
As $A_1 \cap A_2$ is an inert subring of $B$ containing $R$, we have 
$A_1 \cap A_2=R$ by comparing the dimensions. 

(iii) Since $(x,y)B =B$, it follows that $B \neq R^{[2]}$.

(iv) Suppose there exists a $k$-subalgebra $S$ of $B$ such that $B=S^{[2]}$. Then $R=ML(B) \subseteq S$. Since $\td_{k}R
=\td_{k}S$ and both $R$ and $S$ are algebraically closed in $B$, it follows that $R=S$, contradicting (iii). 
Hence the result.
\end{proof}
}
\end{ex}

\begin{rem}
{\em An important problem in Affine Algebraic Geometry asks whether, for the Russell-Koras threefold $B$ defined in 
Example \ref{rkt}, $B^{[1]}={\bC}^{[4]}$. An affirmative answer will give a negative solution to the Zariski Cancellation 
Problem for the three-space in characteristic zero.
It was shown by A. Dubouloz in \cite{DZ} that $ML(B^{[1]})={\bC}$  and A. Dubouloz and J. Fasel have shown that 
$X=Spec(B)$ is ${\A}^{1}$-contractible \cite[Theorem 1.1]{DZF}. This leads to the question:

\medskip
\noindent
{\it {\bf Question 5.9.} Let $B$ be as in Example \ref{rkt}. Is $ML^{*}(B^{[1]})={\bC}~?$}

A negative answer to this problem will confirm that $B$ is not a candidate for the Zariski Cancellation Problem for the 
affine three space.
}
\end{rem}

\noindent
{\small
{\bf Acknowledgement:} The authors thank Professor Amartya K. Dutta for going through the earlier drafts and suggesting 
improvements. The authors also thank the referee for the present version of Lemma 3.1 and other useful suggestions.
The second author also acknowledges Department of Science and Technology for their Swarna Jayanti Award.}
}

{\small{

}}
\end{document}